\documentclass[12pt,reqno]{amsart}
\usepackage[T1]{fontenc}
\usepackage{lmodern}
\usepackage[expansion=false]{microtype}
\usepackage{amsmath,amssymb,amsthm,mathtools}
\usepackage{enumitem}
\usepackage{xcolor}
\usepackage[colorlinks=true,linkcolor=blue!55!black,citecolor=blue!55!black,urlcolor=blue!65!black]{hyperref}
\usepackage[nameinlink,capitalise,noabbrev]{cleveref}
\hypersetup{
  pdftitle={Jacobian layers and higher Milnor sequences: rigidity, stabilization, and log-concavity},
  pdfauthor={Zhiwen Liu},
  pdfsubject={Higher Milnor and Tjurina sequences of isolated hypersurface singularities},
  pdfkeywords={isolated hypersurface singularity, higher Milnor number, Jacobian ideal, semi-homogeneous singularity, log-concavity}
}

\numberwithin{equation}{section}

\newtheorem{theorem}{Theorem}[section]
\newtheorem{proposition}[theorem]{Proposition}
\newtheorem{corollary}[theorem]{Corollary}

\newtheorem{maintheorem}{Main Theorem}

\crefname{maintheorem}{Main Theorem}{Main Theorems}
\Crefname{maintheorem}{Main Theorem}{Main Theorems}
\theoremstyle{definition}
\newtheorem{definition}[theorem]{Definition}
\newtheorem{example}[theorem]{Example}
\theoremstyle{remark}
\newtheorem{remark}[theorem]{Remark}

\newcommand{\C}{\mathbb C}

\newcommand{\m}{\mathfrak m}
\newcommand{\length}{\operatorname{length}}
\newcommand{\ord}{\operatorname{ord}}
\newcommand{\gr}{\operatorname{gr}}
\newcommand{\reg}{\operatorname{reg}}
\newcommand{\inm}{\operatorname{in}_{\m}}

\title[Higher Milnor sequences]{Jacobian layers and higher Milnor sequences:\\rigidity, stabilization, and log-concavity}
\author{Zhiwen Liu}
\address{Beijing Institute of Mathematical Sciences and Applications, Beijing 100084, China}
\email{liuzhiwen@bimsa.cn}

\author{Stephen S.-T. Yau}
\address{Department of Mathematical Sciences, Tsinghua University, Beijing 100084, China}
\email{yau@uic.edu}

\author{Mingyi Zhang}
\address{Primary affiliation:
Center for Mathematics and Interdisciplinary Sciences, Fudan University, Shanghai 200433, China}
\curraddr{
\textsc{Shanghai Institute for Mathematics and Interdisciplinary Sciences (SIMIS), Shanghai 200433, China}}
\email{mingyizhang@simis.cn}

\thanks{Yau is supported by the Tsinghua University Education Foundation.}
\date{}


\begin{document}

\begin{abstract}
Let $f$ be an isolated hypersurface singularity and let $J(f)$ be its Jacobian ideal.
We study the higher Milnor numbers
\[
 \mu_k(f)=\dim_{\C}\mathcal O_n/\m^kJ(f).
\]
Our main result concerns an ordinary semi-homogeneous singularity
$f=f_r+f_{r+1}+\cdots$, where the projective hypersurface defined by $f_r$ is smooth.
We prove that
\[
 \mu_k(f)=\mu_k(f_r)
 \qquad\text{for every }k\geq0.
\]
This gives the equal-weight case of a conjecture of Shi, Yau, and Zuo.
The main algebraic input is the common-order case of the
Valabrega--Valla criterion.  From it, we obtain the graded Jacobian
module and the exact intersection formula.
As a consequence, we obtain a closed formula for $\mu_k$ and determine
the sharp stabilization index of $\m^kJ(f)$.
We also determine the optimal uniform range in which the Tjurina
formula holds for all ordinary semi-homogeneous germs with fixed
\(n\) and \(r\).

We next consider the log-concavity property.
For ordinary semi-homogeneous singularities, the higher Milnor sequence
has an initial log-convex part and a log-concave tail.
For an arbitrary isolated hypersurface singularity in \(n\geq2\)
variables, both the higher Milnor and higher Tjurina sequences are
eventually strictly log-concave.
In the plane curve case, we prove that the log-concavity defects change
sign at most once, and we characterize the sequences which are log-concave
at every positive index.
Finally, we give examples showing that the higher Milnor sequence does not
determine the contact-equivalence class.
\end{abstract}

\maketitle

\section{Introduction}
The study of isolated hypersurface singularities has a long history in which
topological methods and local algebra are closely related.  On the topological
side, Milnor associated a fibration to an isolated hypersurface singularity and
showed that the dimension of the Jacobian algebra is the number of vanishing
cycles, or equivalently the middle Betti number of the Milnor fiber
\cite{Milnor1968}.  On the analytic side, the Tjurina algebra gives the
first-order deformation space of the hypersurface germ; see, for example,
\cite{GreuelLossenShustin2007}.  The Mather--Yau theorem further shows that, over
the complex numbers, the isomorphism type of the Tjurina algebra determines the
 contact-equivalence class \cite{MatherYau1982}.  These classical results show that the local algebras
defined by a function and its partial derivatives contain both geometric and
analytic information.
Let
\[
 R=\mathcal O_n=\C\{x_1,\ldots,x_n\},\qquad
 \m=(x_1,\ldots,x_n),
\]
where \(n\geq2\), and let \(f\in\m^2\) have an isolated critical point at
the origin.  We PJMwrite
\[
 J(f)=(f_{x_1},\ldots,f_{x_n}).
\]
For \(k\geq0\), the higher Milnor and Tjurina numbers are
\begin{equation}\label{eq:def-higher-milnor}
 \mu_k(f)=\dim_{\C}\frac{R}{\m^kJ(f)},\qquad
 \tau_k(f)=\dim_{\C}\frac{R}{(f,\m^kJ(f))}.
\end{equation}
Thus \(\mu_0=\mu\) and \(\tau_0=\tau\).  The classical numbers record only
the total lengths of the Milnor and Tjurina algebras.  By inserting the powers
of the maximal ideal, one keeps the successive infinitesimal layers of the
Jacobian ideal.  In this sense, the two sequences in
\eqref{eq:def-higher-milnor} form an \(\m\)-adic refinement of the classical
invariants.

Higher Tjurina algebras were used by Greuel and Pham in their extension
of the Mather--Yau theorem to positive characteristic.  They proved that a
sufficiently high higher Tjurina algebra determines the germ up to contact
equivalence, and obtained a corresponding statement for higher Milnor algebras and
right equivalence \cite{GreuelPham2017}.  Their result concerns the algebra
structure.  In the present paper we are mainly concerned with the dimensions
of these algebras and with the information contained in the resulting numerical
sequences.

The numerical study was developed further in a series of recent works.
Hussain, Liu, Yau, and Zuo computed the higher Milnor and Tjurina
numbers of binomial plane singularities
\cite{HussainLiuYauZuo2023}.  They also proved that a weighted
homogeneous germ satisfies
\[
 \mu_k(f)-\tau_k(f)=\binom{n+k-2}{n}.
\]

Shi, Yau, and Zuo then studied these
ideals in the framework of \(T\)-maps and considered their behavior under
semi-quasi-homogeneous perturbations \cite{ShiYauZuo2024}.  More recently, Hu,
Yau, and Zuo gave a deformation-theoretic meaning to the higher Tjurina numbers
by introducing deformations with a fixed infinitesimal neighborhood
\cite{HuYauZuo2025}.  In particular, their Main Theorem~A describes the tangent
space of the \(k\)-th deformation functor by
\[
 T\mathbf{Def}^{V(f)}_k
 \simeq \frac{(f,\m^k)}{(f,\m^kJ(f))}.
\]
For weighted homogeneous isolated hypersurface singularities, their
Main Theorem~B gives the initial values of \(\mu_k\) and \(\tau_k\),
and Main Theorem~C expresses the full generating series through
the associated graded Jacobian module.  Main Theorem~D treats weighted
homogeneous germs in three variables.  Related qualitative
questions about these sequences and their associated derivation Lie algebras
were considered in \cite{HussainZuo2026}.

The result of Shi, Yau, and Zuo which is most closely related to the present
paper is \cite[Theorem~4.63]{ShiYauZuo2024}.  If \(f\) is semi-quasi-homogeneous
with principal part \(f_w\), they proved
\[
 \mu_k(f)=\mu_k(f_w)
 \quad\text{for}\quad
 k\leq\min_i\ord(f_{x_i}).
\]
The equality \(\mu(f)=\mu(f_w)\) for the ordinary Milnor number is classical.
For the higher numbers, however, one has to control every layer
\(\m^kJ(f)\), and the problem is therefore more rigid.  Shi, Yau, and Zuo
asked whether the above restriction on \(k\) can be removed and conjectured
that the equality holds for all \(k\)
\cite[Conjecture~4.65]{ShiYauZuo2024}.

We treat the case in which the principal part has equal weights.  After
rescaling the common weight, this is the ordinary homogeneous case.  Namely,
\begin{equation}\label{eq:ordinary-expansion-intro}
 f=f_r+f_{r+1}+f_{r+2}+\cdots,
\end{equation}
where \(f_j\) is homogeneous of degree \(j\) and
\[
 V_+(f_r)\subset\mathbb P^{n-1}
\]
is smooth.  Put
\[
 d=r-1,\qquad
 S=\C[x_1,\ldots,x_n],\qquad
 Q=J(f_r)\subset S.
\]
Then \(Q\) is a complete intersection generated by \(n\) forms of degree
\(d\).

This class is classical and sufficiently rich, but it already contains an
important difficulty:
the homogeneous tangent cone is fixed while the higher terms of the analytic
germ may vary.  It is therefore a useful class for comparing graded geometry
with analytic geometry.  Our results show, on the one hand, that the complete
higher Milnor sequence is rigid in this class and, on the other hand, that this
numerical rigidity does not give analytic classification.  This distinction is
one of the main points of the paper.

Our first result is the following.

\begin{maintheorem}\label{thm:main-a}
Let \(n\geq2\), \(r\geq2\), and let
\[
 f=f_r+f_{r+1}+f_{r+2}+\cdots
 \in\C\{x_1,\ldots,x_n\}
\]
have an isolated critical point at the origin.  Assume that each \(f_j\)
is homogeneous of degree \(j\), that \(f_r\neq0\), and that
\[
 V_+(f_r)\subset\mathbb P^{n-1}
\]
is smooth.  Then
\begin{equation}\label{eq:intro-sqh-invariance}
 \mu_k(f)=\mu_k(f_r)\qquad\text{for all }k\geq0.
\end{equation}
Consequently, the sequence depends only on \(n\) and \(r\).
\end{maintheorem}
This proves the equal-weight case of
\cite[Conjecture~4.65]{ShiYauZuo2024}.  All partial derivatives
\(f_{x_i}\) have order \(d\), and their initial forms form a regular
sequence.  The common-order case of the Valabrega--Valla criterion gives
\[
 J(f)\cap\m^{p+d}=\m^pJ(f).
\]
A direct lift of the Koszul relations also gives
\begin{equation}\label{eq:intro-graded-isomorphism}
 \bigoplus_{p\geq0}\frac{\m^pJ(f)}{\m^{p+1}J(f)}
 \simeq Q(d).
\end{equation}
The intersection formula is a standard filtered-algebra input.  Its use
here identifies every Jacobian layer and removes the bound on \(k\) in
the equal-weight case.

Once the graded module is known, one can determine exactly when the filtered
Jacobian ideal becomes an ordinary power of \(\m\).  We obtain both a closed
formula and a sharp stabilization index.

\begin{maintheorem}
\label{thm:main-b}
Let \(n\geq2\), \(r\geq2\), and let
\[
 f=f_r+f_{r+1}+f_{r+2}+\cdots
 \in\C\{x_1,\ldots,x_n\}
\]
have an isolated critical point.  Assume that each \(f_j\) is homogeneous
of degree \(j\) and that \(V_+(f_r)\subset\mathbb P^{n-1}\) is smooth.
Set
\[
 d=r-1,\qquad K=(n-1)(d-1).
\]
We use the convention \(\binom an=0\) when \(a<n\).  For every \(k\geq0\),
\begin{equation}\label{eq:intro-closed-form}
 \mu_k(f)=d^n+\sum_{j=1}^n(-1)^{j+1}\binom nj
 \binom{n+k-1-(j-1)d}{n}.
\end{equation}
Moreover,
\begin{equation}\label{eq:intro-stabilization}
 \m^kJ(f)=\m^{k+d}\quad\Longleftrightarrow\quad k\geq K,
\end{equation}

If \(K\leq2\), then \(f\in J(f)\), \(\tau_0(f)=d^n\), and
\[
 \tau_k(f)=\mu_k(f)-\binom{n+k-2}{n}
 \qquad(k\geq1).
\]
If \(K\geq3\), the same formula holds for every \(k\geq K-1\).
This uniform cutoff is best possible.  For every smooth homogeneous
principal part \(f_r\), there is a homogeneous polynomial \(h\) of degree
\(r+1\) such that the formula fails for \(F=f_r+h\) at \(k=K-2\).

In particular, for every \(k\geq K\),
\[
 \tau_k(f)=
 \binom{n+d+k-1}{n}-\binom{n+k-2}{n}.
\]
\end{maintheorem}

For a homogeneous principal part \(f_r\), the formula follows from the
Hilbert series of the complete intersection \(S/J(f_r)\).  It is also a direct specialization of the graded Jacobian
construction in \cite[Lemma~6 and Theorem~23]{HuYauZuo2025}.  Related formulas for binomial plane singularities
appear in \cite{HussainLiuYauZuo2023}.  The new point here is that
\Cref{thm:main-a} shows that the formula is unchanged after arbitrary
higher-degree terms are added.  The exact stabilization index and the
optimal uniform Tjurina cutoff are further conclusions of this paper.

We next study the shape of the numerical sequence.  Two related
Milnor-type sequences have opposite behavior.  Teissier's sectional
Milnor numbers form a log-convex sequence, while Huh's projective Milnor
numbers form a log-concave sequence
\cite{Teissier1977Minkowski,Huh2012}.  Both are different from the higher
Milnor sequence considered here.  It is therefore natural to ask whether
the numbers in \eqref{eq:def-higher-milnor} satisfy a similar property.  Put
\[
 \Delta_k=\mu_k^2-\mu_{k-1}\mu_{k+1}.
\]

\begin{maintheorem}\label{thm:main-c}
Let \(f\) be an ordinary semi-homogeneous germ in \(n\geq2\) variables.
Set
\[
 d=\ord(f)-1,\qquad K=(n-1)(d-1).
\]
If \(d=1\), the sequence \((\mu_k)\) is strictly log-concave for all
\(k\geq1\).  Let \(d\geq2\).  Then it is strictly log-convex for
\[
 1\leq k\leq d-1.
\]
If \(n=2\), it is strictly log-concave for \(k\geq d\).  If \(n\geq3\),
it is strictly log-concave for \(k\geq K\).

For an arbitrary isolated hypersurface germ in \(n\geq2\) variables,
both \((\mu_k)\) and \((\tau_k)\) are eventually strictly log-concave.
\end{maintheorem}

The last assertion is a consequence of Hilbert--Samuel theory: the sequences
are eventually polynomial, and their leading terms force strict log-concavity.
We emphasize that this is only an eventual statement.  The
first layers need not be log-concave; for the ordinary germs considered above,
they can even form a strictly log-convex initial part.  Thus the passage from the
Jacobian layers to the Hilbert--Samuel tail is visible in the change of sign of
\(\Delta_k\).

For plane curve singularities, a recent formula of Fern\'andez-P\'erez
and Ribeiro makes this transition completely explicit
\cite[Corollary~4.1]{FernandezPerezRibeiro2025}:
\begin{equation}\label{eq:intro-plane-formula}
 \mu_k=
 \begin{cases}
  \mu+k(k+1),&0\leq k<r,\\[2pt]
  \displaystyle\mu+k(k+1)-\binom{k-r+2}{2},&k\geq r,
 \end{cases}
\end{equation}
where \(r=\ord(f)\) and \(\mu=\mu_0(f)\).  As an application of this
formula, we obtain a complete description of the change from
log-convexity to log-concavity.

\begin{maintheorem}\label{thm:main-d}
For an isolated plane curve singularity, the defects
\(\Delta_k=\mu_k^2-\mu_{k-1}\mu_{k+1}\) change sign at most once, from
negative to nonnegative.  The first nonnegative index is given explicitly
in \Cref{prop:plane-defect}.

The full sequence is log-concave at every positive index if and only if
\[
 r=2\quad\text{and}\quad\mu\leq4.
\]
It is strictly log-concave at every positive index if and only if
\[
 r=2\quad\text{and}\quad\mu\leq3.
\]
The sequence determines the pair \((\mu,r)\), but it does not determine
the contact-equivalence class.
\end{maintheorem}

Formula \eqref{eq:intro-plane-formula} shows that
the higher-Milnor claims in Conjecture~1.10 in
\cite{HussainZuo2026} doesn't hold.  A separate computation for the \(E_6\)
singularity gives a counterexample to the higher-Tjurina claims in the Conjecture~1.10.

Let us say a few words about the proofs.  Since the projective hypersurface
\(V_+(f_r)\) is smooth, the partial derivatives of
\(f_r\) form a regular sequence.  The common-order case of the
Valabrega--Valla criterion gives the intersection formula, while a direct
lift of the Koszul relations identifies the graded Jacobian module.
The formula for \(\mu_k\) follows
from the Hilbert series
\[
 H_{S/Q}(t)=(1+t+\cdots+t^{d-1})^n.
\]
The stabilization index is the first integer \(k\) for which \(d+k\) is
larger than the socle degree \(s=n(d-1)\) of \(S/Q\).  The sharp Tjurina
range follows from its Gorenstein duality.  The
plane case is obtained by a direct calculation from
\eqref{eq:intro-plane-formula}, while eventual log-concavity in the general case
comes from the Hilbert--Samuel polynomials.

Finally, we discuss the amount of geometric information contained in the
sequence \((\mu_k)\).  At first sight, \Cref{thm:main-a} may suggest a strong
classification statement.  However, the sequence records only the dimensions
of the graded pieces and forgets their multiplication.  We give families with
different tangent cones, or with different higher-order deformation directions,
but with the same higher Milnor sequence.  Consequently, the rigidity proved
here should be regarded as a numerical rigidity statement rather than a
Torelli-type result.  This also indicates why
the higher algebras may contain substantially more information than their
dimensions.

The paper is organized as follows.  After the preliminary discussion in
\Cref{sec:preliminaries}, we begin in \Cref{sec:layers} with the
Jacobian layers and their Hilbert--Samuel interpretation.  In
\Cref{sec:strict-jacobian}, we study the graded Jacobian module of an
ordinary semi-homogeneous singularity and prove \Cref{thm:main-a}.  We then
derive the formulas and the sharp stabilization results in
\Cref{sec:formula-stabilization}.  The next two sections are devoted to
log-concavity.  We first discuss the general log-convex and log-concave
phases in \Cref{sec:log-phases}, and then give a more precise description
for plane curve singularities in \Cref{sec:plane}.  Finally,
\Cref{sec:geometry-loss} gives examples showing that the same numerical
sequence may come from different geometric situations.

\section{Preliminaries and motivation}\label{sec:preliminaries}

This section fixes the local and graded notation used in the paper.  It also
places the higher Milnor and Tjurina numbers in the setting of filtered
modules and Hilbert--Samuel functions.

\subsection{The local ring and the classical invariants}

We work in the local ring
\[
 R=\mathcal O_n=\C\{x_1,\ldots,x_n\}.
\]
Its maximal ideal is
\[
 \m=(x_1,\ldots,x_n).
\]
The ideal \(\m^k\) consists of germs whose terms have degree at least \(k\).
Thus \(R/\m^k\) records Taylor expansions up to degree \(k-1\).  If an
\(R\)-module \(N\) has finite length, then
\[
 \length_R(N)=\dim_{\C}N.
\]
For a nonzero germ \(h=\sum_{j\geq0}h_j\), with each \(h_j\) homogeneous of
degree \(j\), its order is
\[
 \ord(h)=\min\{j:h_j\neq0\}.
\]
Thus \(h\in\m^k\) exactly when \(\ord(h)\geq k\).

Let \(f\in\m^2\).  Its Jacobian ideal is
\[
 J(f)=(f_{x_1},\ldots,f_{x_n}).
\]
The origin is an isolated critical point exactly when \(J(f)\) is
\(\m\)-primary.  In this case, the Milnor algebra
\[
 M(f)=R/J(f)
\]
has finite dimension.  Its dimension is the Milnor number \(\mu(f)\).
The Milnor number is also the middle Betti number of the Milnor fiber
\cite{Milnor1968}.

The Tjurina algebra is
\[
 T(f)=R/(f,J(f)).
\]
Its dimension is the Tjurina number \(\tau(f)\).  This algebra describes
first-order deformations of the hypersurface germ; see
\cite{GreuelLossenShustin2007}.  The two classical numbers have different
roles.  The Milnor number has a topological meaning.  The Tjurina number is
tied more closely to analytic deformation.

Higher Tjurina algebras also occur in higher Mather--Yau questions and in
deformation problems with a fixed infinitesimal neighborhood
\cite{GreuelPham2017,HuYauZuo2025}.

\subsection{Higher quotients and the layer dictionary}\label{sec:layers}

Recall from \eqref{eq:def-higher-milnor} that
\[
 \mu_k(f)=\length_R\frac{R}{\m^kJ(f)},\qquad
 \tau_k(f)=\length_R\frac{R}{(f,\m^kJ(f))}.
\]
Since \(J(f)\) is \(\m\)-primary, these quotients have finite length.

For \(k\geq0\), the ideal \(\m^kJ(f)\) consists of sums
\[
 \sum_{i=1}^n a_i f_{x_i},\qquad a_i\in\m^k.
\]
The index \(k\) records the order to which the coefficients vanish.  The
inclusions
\[
 \m^{k+1}J(f)\subseteq\m^kJ(f)
\]
show that \((\mu_k)\) and \((\tau_k)\) are nondecreasing.  Their zeroth
terms are the usual Milnor and Tjurina numbers.

The Milnor layer formulas below are also used in
\cite[Section~3.2]{HuYauZuo2025}.

\begin{proposition}\label{prop:layer-dictionary}
Let \(J=J(f)\).  For every \(k\geq0\),
\begin{align}
 \mu_k&=\mu_0+\length_R(J/\m^kJ),\label{eq:mu-dictionary}\\
 \mu_{k+1}-\mu_k
 &=\length_R(\m^kJ/\m^{k+1}J).\label{eq:mu-increment}
\end{align}
Set \(A=R/(f)\), let \(\bar\m=\m A\), and let
\(\bar J=(J+(f))/(f)\subset A\).  Then
\begin{align}
 \tau_k&=\tau_0+\length_A(\bar J/\bar\m^k\bar J),
 \label{eq:tau-dictionary}\\
 \tau_{k+1}-\tau_k
 &=\length_A(\bar\m^k\bar J/\bar\m^{k+1}\bar J).
 \label{eq:tau-increment}
\end{align}
\end{proposition}

\begin{proof}
The first identity follows from the exact sequence
\[
 0\longrightarrow J/\m^kJ\longrightarrow R/\m^kJ
 \longrightarrow R/J\longrightarrow0.
\]
Taking the difference for consecutive values of \(k\) gives
\eqref{eq:mu-increment}.  The same argument in \(A\) gives
\eqref{eq:tau-dictionary} and \eqref{eq:tau-increment}.
\end{proof}

Let \(M\) be a finitely generated module over a local ring with maximal
ideal \(\mathfrak n\).  Its associated graded module is
\[
 \gr_{\mathfrak n}(M)
 =\bigoplus_{p\geq0}\frac{\mathfrak n^pM}{\mathfrak n^{p+1}M}.
\]
The Hilbert function of this graded module is
\[
 h_M(p)=\length\frac{\mathfrak n^pM}{\mathfrak n^{p+1}M}.
\]
Thus \eqref{eq:mu-increment} says that
\(\mu_{k+1}-\mu_k=h_J(k)\).  Equation \eqref{eq:tau-increment} gives the
same interpretation for \(\tau_{k+1}-\tau_k\), using \(\bar J\) over
\(A\).  The higher sequences are therefore cumulative sums of the Hilbert
functions of these two filtered modules.

\subsection{Associated graded objects and the two filtrations}

The layer dictionary uses the coefficient filtration
\[
 F_pJ=\m^pJ.
\]
The inclusion \(J\subset R\) gives a second filtration,
\[
 G_qJ=J\cap\m^q.
\]
The first filtration gives the differences of the higher Milnor numbers.
The second gives the form ideal of \(J\).

The associated graded ring of \(R\) is
\[
 \gr_{\m}(R)\simeq S=\C[x_1,\ldots,x_n].
\]
The initial form of a nonzero germ is its lowest homogeneous part.  The
initial forms of all elements of an ideal \(I\) generate its form ideal
\(\inm(I)\subset S\).  Under the displayed identification,
\[
 \inm(I)_q\simeq\frac{I\cap\m^q}{I\cap\m^{q+1}}.
\]
There is also a natural graded isomorphism
\[
 \gr_{\m}(R/I)\simeq S/\inm(I).
\]
This quotient formula always holds.  It becomes effective once the form
ideal is known, because graded ideals can be studied by Hilbert series and
free resolutions.

\subsection{The tangent cone and the ordinary case}

Let \(r=\ord(f)\), and write
\[
 f=f_r+f_{r+1}+f_{r+2}+\cdots,
\]
where each \(f_j\) is homogeneous of degree \(j\).  The polynomial \(f_r\)
is the lowest homogeneous part of \(f\).  The hypersurface
\(V(f_r)\subset\C^n\) is the tangent cone of \(V(f)\) at the origin.  Its
projectivization
\[
 X=V_+(f_r)\subset\mathbb P^{n-1}
\]
records the nonzero tangent directions.

\begin{definition}\label{def:ordinary}
The germ \(f\) is \emph{ordinary semi-homogeneous of multiplicity \(r\)} if
it has an isolated critical point and \(X\) is smooth.  We set
\[
 d=r-1,\qquad g_i=f_{x_i},\qquad q_i=(f_r)_{x_i},
\]
and write
\[
 J=(g_1,\ldots,g_n)\subset R,\qquad
 Q=(q_1,\ldots,q_n)\subset S.
\]
\end{definition}

Euler's identity shows that a common projective zero of
\(q_1,\ldots,q_n\) would be a singular point of \(X\).  Hence these forms
have no common projective zero.  The ideal \(Q\) is primary to
\((x_1,\ldots,x_n)\), and the \(q_i\) form a regular sequence.  Equivalently,
their Koszul complex is exact.  Thus \(S/Q\) is an Artinian complete
intersection.  Its Hilbert series depends only on \(n\) and \(d\); see
\eqref{eq:jacobian-ring-Hilbert-series}.

Each \(g_i\) has order \(d\) and initial form \(q_i\).  The main comparison
needed below is
\[
 J\cap\m^{p+d}=\m^pJ\qquad(p\geq0).
\]
The common-order case of the Valabrega--Valla criterion gives this
equality.  The next section also gives a direct proof by lifting the
Koszul relations.  The direct proof identifies the coefficient-graded
module of \(J\) with a shift of \(Q\).

\subsection{Hilbert--Samuel growth and log-concavity}

Let \((B,\mathfrak n)\) be a Noetherian local ring, and let \(M\) be a
finitely generated \(B\)-module.  Its Hilbert--Samuel function is
\[
 H_M(k)=\length_B(M/\mathfrak n^kM).
\]
It is the cumulative sum of the layer function:
\[
 H_M(k)=\sum_{p=0}^{k-1}
 \length_B\frac{\mathfrak n^pM}{\mathfrak n^{p+1}M}.
\]
If \(M\) has dimension \(e\), then \(H_M(k)\) agrees with a polynomial of
degree \(e\) for all large \(k\), and
\[
 H_M(k)=\frac{e_{\mathfrak n}(M)}{e!}k^e+O(k^{e-1}).
\]
Here \(e_{\mathfrak n}(M)\) is the Hilbert--Samuel multiplicity; see
\cite{Matsumura1986,BrunsHerzog1993}.

In the notation of \Cref{prop:layer-dictionary},
\[
 \mu_k-\mu_0=H_J(k),\qquad
 \tau_k-\tau_0=H_{\bar J}(k),
\]
where the second function is taken over \(A=R/(f)\).  The module \(J\) has
dimension \(n\).  Since \(A/\bar J\) has finite length, \(\bar J\) has
dimension \(n-1\).  This gives the different polynomial degrees of the two
sequences and explains their eventual behavior.

For a positive sequence \((a_k)\), log-concavity at \(k\) means
\[
 a_k^2\geq a_{k-1}a_{k+1}.
\]
Log-convexity uses the reverse inequality.  For a positive sequence,
log-concavity is equivalent to the ratios \(a_{k+1}/a_k\) being
nonincreasing.  If \(P\) has positive degree and positive leading
coefficient, then \((P(k))\) is strictly log-concave for all large \(k\).
Hilbert--Samuel theory therefore controls the tail of the sequences.  The
early and middle terms depend on the individual Jacobian layers.

The layer dictionary gives the local meaning of the numerical sequences.
The ordinary condition makes the layers accessible through the tangent
cone, while Hilbert--Samuel theory describes their polynomial tail.  The
later sections use these links to study rigidity, stabilization, and the
change from the first layers to the tail.  They also compare the numerical
sequence with the geometric data of the singularity.

\section{The graded Jacobian module of an ordinary multiple point}
\label{sec:strict-jacobian}

We keep the notation of \Cref{def:ordinary}.

The forms $q_1,\ldots,q_n$ have degree $d$.
They have no common projective zero.
Thus $Q$ is $(x_1,\ldots,x_n)$-primary, and the $q_i$ form a regular sequence in $S$.

\begin{proposition}
\label{prop:strict-jacobian}
Let $f$ be as in \Cref{def:ordinary}.
Give $J$ the coefficient filtration $F_pJ=\m^pJ$.
There is a natural graded $S$-module isomorphism
\begin{equation}\label{eq:strict-main-isom}
 \mathcal G(J):=\bigoplus_{p\geq0}\frac{\m^pJ}{\m^{p+1}J}
 \simeq Q(d),\qquad Q(d)_p=Q_{p+d}.
\end{equation}
For every $p\geq0$, one also has
\begin{equation}\label{eq:strict-intersection}
 J\cap\m^{p+d}=\m^pJ.
\end{equation}
In particular,
\begin{equation}\label{eq:form-milnor-algebra}
 \gr_{\m}(R/J)\simeq S/Q.
\end{equation}
\end{proposition}

\begin{proof}
The initial-syzygy presentation used below also appears in
\cite[Lemma~9]{HuYauZuo2025}.
We give the argument in the present analytic setting and then
use the regular sequence $q_1,\ldots,q_n$ to identify
the initial syzygy module.
Let $Z(g)\subset R^n$ be the module of first syzygies among $g_1,\ldots,g_n$.
Let $Z(q)\subset S^n$ be the corresponding homogeneous syzygy module for $q_1,\ldots,q_n$.

There is a graded surjection
\begin{equation}\label{eq:surjection-to-GJ}
 S^n\longrightarrow\mathcal G(J).
\end{equation}
In degree $p$, it sends a homogeneous vector $(\bar a_1,\ldots,\bar a_n)$ to the class of
$\sum a_i g_i$ in $\m^pJ/\m^{p+1}J$.
Here each $a_i\in\m^p$ lifts $\bar a_i$.
The kernel is the coefficientwise initial module $\inm Z(g)$.
Indeed, if the displayed class is zero, then
\[
 \sum_i a_i g_i=\sum_i b_i g_i
\]
for some $b_i\in\m^{p+1}$.
The vector $(a_i-b_i)$ is a syzygy, and its leading vector is $(\bar a_i)$.
The converse is immediate.

All $g_i$ have order $d$, with initial form $q_i$.
The leading vector of a syzygy among the $g_i$ is therefore a syzygy among the $q_i$.
This gives
\begin{equation}\label{eq:one-inclusion-syz}
 \inm Z(g)\subseteq Z(q).
\end{equation}
For the reverse inclusion, use the regular sequence $q_1,\ldots,q_n$.
Its Koszul complex is exact.
Hence every homogeneous element of $Z(q)$ is a homogeneous sum of the relations
\[
 q_j e_i-q_i e_j.
\]
Replace $q_i$ by $g_i$ in the same sum.
The result is an exact Koszul relation among the $g_i$.
Its leading vector is the given relation among the $q_i$.
Thus
\[
 \inm Z(g)=Z(q).
\]
Taking the quotient in \eqref{eq:surjection-to-GJ} gives
\[
 \mathcal G(J)\simeq S^n/Z(q)\simeq Q(d),
\]
which proves \eqref{eq:strict-main-isom}.

Under this isomorphism, the degree-$p$ class of an element of $\m^pJ$ is its homogeneous form of degree $p+d$.
Injectivity gives
\begin{equation}\label{eq:one-step-intersection}
 \m^pJ\cap\m^{p+d+1}=\m^{p+1}J.
\end{equation}
We now prove \eqref{eq:strict-intersection} by induction on $p$.
For $p=0$, it follows from $J\subset\m^d$.
If it holds for $p$, then \eqref{eq:one-step-intersection} gives
\[
 J\cap\m^{p+d+1}
 =(J\cap\m^{p+d})\cap\m^{p+d+1}
 =\m^pJ\cap\m^{p+d+1}
 =\m^{p+1}J.
\]
This proves the induction.
The form ideal of $J$ is now $Q$, so \eqref{eq:form-milnor-algebra} follows.
\end{proof}

\begin{remark}\label{rem:filtration-warning}
Apply \cite[Corollary~2.7]{ValabregaValla1978} with filtration ideal
\(\mathfrak a=\m\) and \(\ord_{\m}(g_i)=d\).  Since the initial forms
\(q_1,\ldots,q_n\) form a regular sequence, the criterion gives
\[
 J\cap\m^a=\sum_i\m^{a-d}g_i
 \qquad(a\geq d).
\]
Taking \(a=p+d\) gives
\[
 J\cap\m^{p+d}=\m^pJ.
\]
The direct proof above also records the coefficient-graded module
isomorphism used later.
\end{remark}

\begin{proof}[Proof of \Cref{thm:main-a}]
The complete intersection $Q$ is generated by $n$ forms of degree $d$.
By \eqref{eq:form-milnor-algebra},
\[
 \mu_0(f)=\length_R(R/J)=\dim_{\C}(S/Q)=d^n.
\]
By \eqref{eq:strict-main-isom},
\[
 \length_R(J/\m^kJ)=\sum_{p=0}^{k-1}\dim_{\C}Q_{p+d}.
\]
The same two expressions hold for the homogeneous germ $f_r$.
Now use \eqref{eq:mu-dictionary}.
\end{proof}

Shi, Yau, and Zuo prove in their Theorem~4.63 that
\(\mu_k(f)=\mu_k(f_w)\) for
\[
 k\leq\min_i\ord(f_{x_i})
\]
for a semi-quasi-homogeneous series with principal part $f_w$.
Their Conjecture~4.65 asks whether equality holds for every
$k\geq0$ \cite{ShiYauZuo2024}.
The weights in \Cref{thm:main-a} are $(1,\ldots,1)$.
Thus \Cref{thm:main-a} proves the equal-weight case of their conjecture.
The proof uses the common ordinary order of all partial derivatives.
This common order is absent for general weights.

\begin{corollary}\label{cor:resolution-regularity}
The minimal graded resolution of $\mathcal G(J)$ is the truncated Koszul complex
\begin{equation}\label{eq:truncated-koszul-resolution}
 \begin{aligned}
 0\longrightarrow S(-(n-1)d)&\longrightarrow\cdots
 \longrightarrow S(-2d)^{\binom n3}\\
 &\longrightarrow S(-d)^{\binom n2}
 \longrightarrow S^n\longrightarrow\mathcal G(J)\longrightarrow0.
 \end{aligned}
\end{equation}
Consequently,
\begin{equation}\label{eq:regularity-GJ}
 \reg_S\mathcal G(J)=(n-1)(d-1).
\end{equation}
\end{corollary}

\begin{proof}
The ideal $Q$ has the Koszul resolution of a regular sequence of degree $d$.
Shift that resolution by $d$ and use \eqref{eq:strict-main-isom}.
At homological degree $i$, the shift is $id$.
The largest value of $id-i$ is $(n-1)(d-1)$.
\end{proof}

\section{Milnor formulas and sharp stabilization for ordinary
semi-homogeneous germs}
\label{sec:formula-stabilization}

For the homogeneous germ $f_r$, the Hilbert-series calculation
below is the equal-degree specialization of
\cite[Lemma~6 and Theorem~23]{HuYauZuo2025}.
Proposition~3.1 gives the same graded Jacobian module
for an ordinary semi-homogeneous germ.
We record the resulting formulas, derive the stabilization
index from the socle degree, and determine the optimal
uniform Tjurina cutoff.

Let
\[
 A_X=S/Q
\]
be the Jacobian ring of the smooth tangent-cone hypersurface
$X=V_+(f_r)\subset\mathbb P^{n-1}$.
Set
\[
 h_a=\dim_{\C}(A_X)_a,
 \qquad s=n(d-1).
\]
Since $Q$ is a complete intersection of type $(d,\ldots,d)$,
\begin{equation}\label{eq:jacobian-ring-Hilbert-series}
 H_{A_X}(t)=\sum_{a\geq0}h_at^a
 =\frac{(1-t^d)^n}{(1-t)^n}
 =(1+t+\cdots+t^{d-1})^n.
\end{equation}
In particular, $s$ is the socle degree and $h_s=1$.

\begin{proposition}
\label{prop:universal-formula}
Let \(f\) be ordinary semi-homogeneous as in \Cref{def:ordinary}.
For every $k\geq0$,
\begin{align}
 \mu_k(f)
 &=\binom{n+d+k-1}{n}+\sum_{a=d+k}^{s}h_a,
 \label{eq:tail-formula}\\
 &=d^n+\sum_{j=1}^n(-1)^{j+1}\binom nj
 \binom{n+k-1-(j-1)d}{n}.
 \label{eq:binomial-formula}
\end{align}
Here an empty sum is zero and $\binom an=0$ for $a<n$.
More precisely, put $D=d+k$ and
\[
 (A_X)_{\geq D}:=\bigoplus_{a\geq D}(A_X)_a.
\]
Give \(\m^{d+k}/\m^kJ\) its \(\m\)-adic filtration.  With the shift
convention \(M(D)_\ell=M_{D+\ell}\), there is an isomorphism of graded
\(S\)-modules
\begin{equation}\label{eq:graded-deficit}
 \gr_{\m}\left(\frac{\m^{d+k}}{\m^kJ}\right)
 \simeq (A_X)_{\geq d+k}(d+k).
\end{equation}
The generating function is
\begin{equation}\label{eq:mu-generating-function}
 \sum_{k\geq0}\mu_k(f)z^k
 =\frac{d^n}{1-z}
 +\frac{\displaystyle\sum_{j=1}^n(-1)^{j+1}\binom nj
 z^{(j-1)d+1}}{(1-z)^{n+1}}.
\end{equation}
\end{proposition}

\begin{proof}
By \Cref{prop:strict-jacobian},
\begin{align*}
 H_{\mathcal G(J)}(t)
 &=t^{-d}H_Q(t)\\
 &=t^{-d}\frac{1-(1-t^d)^n}{(1-t)^n}\\
 &=\sum_{j=1}^n(-1)^{j+1}\binom nj
 \frac{t^{(j-1)d}}{(1-t)^n}.
\end{align*}
The coefficient of $t^p$ is
\[
 \length_R(\m^pJ/\m^{p+1}J).
\]
Sum these coefficients for $0\leq p<k$, add $\mu_0=d^n$, and use
\[
 \sum_{p=0}^{k-1}\binom{n+p-1-c}{n-1}
 =\binom{n+k-1-c}{n}.
\]
This proves \eqref{eq:binomial-formula}.
Summing once more in $k$ gives \eqref{eq:mu-generating-function}.

We next prove the stronger statement \eqref{eq:graded-deficit}.
By \eqref{eq:strict-intersection},
\[
 \m^kJ=J\cap\m^{d+k}.
\]
Set $D=d+k$ and $N_k=\m^D/\m^kJ$.  For every $\ell\geq0$,
\[
 \m^\ell N_k
 =\frac{\m^{D+\ell}+\m^kJ}{\m^kJ}.
\]
Since \(D+\ell\geq D\), the intersection theorem gives
\[
 \m^kJ\cap\m^{D+\ell}
 =(J\cap\m^D)\cap\m^{D+\ell}
 =J\cap\m^{D+\ell}
 =\m^{k+\ell}J.
\]
Taking degree-\(D+\ell\) initial forms and using
\eqref{eq:strict-main-isom}, we obtain
\[
 \left(\gr_{\m}N_k\right)_\ell
 \simeq\frac{S_{D+\ell}}{Q_{D+\ell}}
 =(A_X)_{D+\ell}.
\]
These degreewise isomorphisms are compatible with the \(S\)-action.
They therefore define the shifted graded \(S\)-module isomorphism
\eqref{eq:graded-deficit}.
Taking dimensions and using
\[
 \length_R(R/\m^{d+k})=\binom{n+d+k-1}{n}
\]
gives \eqref{eq:tail-formula}.
\end{proof}

Formula \eqref{eq:tail-formula} has a useful geometric form.
The first term is the length of the ambient fat point of order $d+k$.
The second term is a tail of the Jacobian ring of the smooth tangent cone.
The Hilbert function of this ring depends only on $n$ and $r$.
The Hilbert function alone does not record the graded algebra structure
of \(A_X\).  Griffiths' residue theory relates suitable graded pieces of
the Jacobian ring to the primitive cohomology of \(X\); see
\cite{Griffiths1969a,Griffiths1969b}.

\begin{corollary}\label{cor:sharp-stabilization}
Let \(f\) be ordinary semi-homogeneous as in \Cref{def:ordinary}, and set
\begin{equation}\label{eq:def-K}
 K=(n-1)(d-1).
\end{equation}
Then
\begin{equation}\label{eq:sharp-ideal-equality}
 \m^kJ=\m^{d+k}\quad\Longleftrightarrow\quad k\geq K.
\end{equation}
If $K\geq1$, the failure at the preceding index is one-dimensional:
\begin{equation}\label{eq:last-deficit-one}
 \length_R\left(\frac{\m^{d+K-1}}{\m^{K-1}J}\right)=1.
\end{equation}
Moreover, \(K=\reg_S\mathcal G(J)\).
\end{corollary}

\begin{proof}
By \eqref{eq:graded-deficit}, equality in \eqref{eq:sharp-ideal-equality}
holds exactly when $(A_X)_a=0$ for every $a\geq d+k$.
The largest nonzero degree of $A_X$ is $s=n(d-1)$.
Thus equality holds exactly when
\[
 d+k>s,
\]
which is $k\geq(n-1)(d-1)$.
For $k=K-1$, the only remaining graded part starts in degree
$d+K-1=s$, and $h_s=1$.
The last statement follows from \eqref{eq:regularity-GJ}.
\end{proof}

For a weighted homogeneous germ, the identity
\[
\mu_k-\tau_k=\binom{n+k-2}{n},
\qquad k\ge1,
\]
is known; see \cite{HussainLiuYauZuo2023} and
\cite[Section~3.1]{HuYauZuo2025}.
We now determine the optimal range in which this identity
holds uniformly after higher-degree terms are added
to a smooth homogeneous principal part.

\begin{proposition}
\label{prop:universal-tjurina-tail}
Let \(f\) be ordinary semi-homogeneous as in \Cref{def:ordinary}, and set
\(K=(n-1)(d-1)\).
If \(K\leq2\), then \(f\in J(f)\) and the whole Tjurina sequence depends
only on \((n,r)\):
\begin{equation}\label{eq:universal-tjurina-all}
 \tau_0(f)=d^n,\qquad
 \tau_k(f)=\mu_k(f)-\binom{n+k-2}{n}\quad(k\geq1).
\end{equation}
If \(K\geq3\), then the second formula holds for every \(k\geq K-1\).
The cutoff \(K-1\) is best possible uniformly over this class.  For
every smooth homogeneous form \(f_r\), there is a homogeneous form \(h\)
of degree \(r+1\) such that \(F=f_r+h\) does not satisfy the formula at
\(k=K-2\).

In particular, for every \(k\geq K\),
\begin{equation}\label{eq:universal-tjurina-tail}
 \tau_k(f)=
 \binom{n+d+k-1}{n}-\binom{n+k-2}{n}.
\end{equation}
\end{proposition}

\begin{proof}
For \(k\geq1\), the intersection theorem gives
\begin{equation}\label{eq:tjurina-colon}
 (\m^kJ:f)=(J:f)\cap\m^{k-1}.
\end{equation}
Indeed, the left-hand side is contained in \(J:f\), and
\(\ord(af)\geq k+d\) forces \(\ord(a)\geq k-1\).  Conversely,
\[
 a\in(J:f)\cap\m^{k-1}
 \Longrightarrow
 af\in J\cap\m^{k+d}=\m^kJ.
\]

Degreewise Euler gives
\[
 e:=f-\frac1r\sum_i x_if_{x_i}\in\m^{r+1}=\m^{d+2}.
\]
The associated graded Milnor algebra has last nonzero degree
\(s=n(d-1)\), so \(\m^{s+1}(R/J)=0\), and \(s+1=K+d\).
If \(K\leq2\), then \(d+2\geq s+1\), whence \([f]=[e]=0\) in \(R/J\).
Thus \(f\in J\).  If \(K\geq3\), then
\[
 \m^{K-2}[f]\subseteq\m^{K-2+d+2}(R/J)
 =\m^{s+1}(R/J)=0.
\]
Equation \eqref{eq:tjurina-colon} therefore becomes
\((\m^kJ:f)=\m^{k-1}\) in every asserted positive-index range.
The multiplication sequence
\[
 0\longrightarrow R/\m^{k-1}
 \xrightarrow{\ \cdot f\ }R/\m^kJ
 \longrightarrow R/(f,\m^kJ)\longrightarrow0
\]
proves \eqref{eq:universal-tjurina-all}.  The case \(k=0,K\leq2\)
follows from \(f\in J\).  Combining this with stabilization proves
\eqref{eq:universal-tjurina-tail}.

It remains to prove sharpness.  Suppose \(K\geq3\), and set
\(A_X=S/J(f_r)\).  Since \(d+2\leq s\), choose a homogeneous \(h\) of
degree \(d+2=r+1\) with nonzero class in \((A_X)_{d+2}\).
The Artinian Gorenstein perfect pairing supplies a homogeneous \(a\) of
degree \(s-(d+2)=K-3\) such that \([ah]\ne0\) in the socle.
For \(F=f_r+h\), Euler's identity gives
\[
 [F]=-\frac1r[h]\quad\text{in }R/J(F).
\]
The associated graded algebra is \(A_X\), so the nonzero class
\([ah]\) shows that \(aF\notin J(F)\).  Hence at \(k=K-2\),
\[
 (\m^kJ(F):F)\subsetneq\m^{K-3}.
\]
For the homogeneous germ \(f_r\), Euler gives \(f_r\in J(f_r)\), so the
corresponding colon equals \(\m^{K-3}\).  The homogeneous germ \(f_r\)
satisfies the stated Tjurina formula.  The germs \(f_r\) and \(F\) have
the same higher Milnor numbers but different Tjurina numbers at
\(k=K-2\).  Hence the formula fails for \(F\) at that index.  No smaller
cutoff can therefore work uniformly.  Together with
\Cref{prop:universal-formula,cor:sharp-stabilization}, this proves
\Cref{thm:main-b}.
\end{proof}

\begin{remark}\label{rem:tau-low-varies}
Consider
\[
 f_0=x^5+y^5,
 \qquad
 f_1=x^5+y^5+x^3y^3.
\]
Both are ordinary semi-homogeneous with the same degree-five tangent cone.
Hence all their higher Milnor numbers agree.
However,
\[
 (\tau_0(f_0),\tau_1(f_0))=(16,18),
 \qquad
 (\tau_0(f_1),\tau_1(f_1))=(15,17).
\]
For $f_1$, the identity
\[
 x(f_1)_x+y(f_1)_y-5f_1=x^3y^3
\]
shows that the class of $f_1$ in its Milnor algebra is a multiple of
$x^3y^3$.  We record why this class is nonzero and spans a socle ideal.
Writing $c=3/5$, the Jacobian relations are
\[
 x^4=-cx^2y^3,\qquad y^4=-cx^3y^2.
\]
For $a=x^4y^3$, repeated substitution gives $a=-c^5y^5a$.
The factor $1+c^5y^5$ is a unit, so $x^4y^3=0$; symmetrically,
$x^3y^4=0$.  Hence $x^3y^3$ is a socle element.  It is nonzero because
its initial form survives in
\(\C[x,y]/(x^4,y^4)\simeq\gr_{\m}(R/J(f_1))\).
Thus quotienting by $f_1$ removes one dimension.  Moreover
\((J(f_1):f_1)=\m\), so \eqref{eq:tjurina-colon} at \(k=1\) removes one
dimension from \(\mu_1=18\) and gives \(\tau_1(f_1)=17\).  For \(f_0\),
Euler gives \(f_0\in J(f_0)\), hence \(\tau_1(f_0)=\mu_1(f_0)=18\).
Here \(K=3\), and both sequences agree from \(k=2=K-1\) onward by
\Cref{prop:universal-tjurina-tail}.  Thus the optimal cutoff in the proposition
is already sharp in the smallest possible case.
\end{remark}

\begin{example}\label{ex:cubic-cone-sequence}
Let $n=3$ and $r=3$, so $d=2$ and $K=2$.
Every ordinary triple surface point has
\[
 (\mu_0,\mu_1,\mu_2,\mu_3,\mu_4,\ldots)
 =(8,11,20,35,56,\ldots).
\]
At $k=2$, the Jacobian filtration has already stabilized:
\[
 \m^2J(f)=\m^4.
\]
The projective tangent cone is a smooth plane cubic.
Its isomorphism class is not visible in the displayed sequence.
\end{example}

\begin{remark}\label{rem:syz-correction}
Shi, Yau, and Zuo consider the homogeneous cubic
\[
 h=x^2y+y^2z+z^2x
\]
and a semi-homogeneous deformation in \cite{ShiYauZuo2024}.
The sequence printed in their Example 4.64(3) begins $(8,14,24,39,\ldots)$.
Its partials are
\[
 2xy+z^2,\qquad x^2+2yz,\qquad y^2+2zx.
\]
A common projective zero cannot have a zero coordinate; if all coordinates
were nonzero, multiplying the three equations would give \(1=(-2)^3\).
Thus the projective cubic is smooth, and
\Cref{prop:universal-formula} gives instead
\[
 (\mu_k(h))_{k\geq0}=(8,11,20,35,56,84,120,\ldots).
\]
The same corrected sequence holds for the deformation.
Already the first-layer formula forces $\mu_1=8+3=11$ and $\mu_2=8+12=20$.
\end{remark}

\section{Log-convex and log-concave phases}\label{sec:log-phases}

In this section we prove \Cref{thm:main-c}.

For a positive sequence $(a_k)$, define its log-concavity defect by
\[
 \Delta_k(a)=a_k^2-a_{k-1}a_{k+1}.
\]
The sequence is log-concave at $k$ when $\Delta_k(a)\geq0$.
It is log-convex at $k$ when $\Delta_k(a)\leq0$.
We write $\Delta_k$ for the defect of $(\mu_k)$.
Two related Milnor-type sequences have opposite behavior.
Teissier's sectional Milnor numbers form a log-convex sequence, while
Huh's projective Milnor numbers form a log-concave sequence
\cite{Teissier1977Minkowski,Huh2012}.  Both are finite sequences indexed
by dimension.  They are different from the sequence \((\mu_k)\) in
\eqref{eq:def-higher-milnor}, where \(k\) records the power of the
maximal ideal.  Our results show that one higher Milnor sequence can
have an initial log-convex part and a log-concave tail.

\begin{proposition}\label{prop:two-log-phases}
Let $f$ be an ordinary semi-homogeneous germ in $n\geq2$ variables.
If $d=1$, then $(\mu_k)$ is strictly log-concave for every $k\geq1$.
Suppose $d\geq2$.
For $1\leq k\leq d-1$,
\begin{equation}\label{eq:early-defect-n}
 \Delta_k
 =n\binom{k+n-2}{n-2}
 \left(
 \frac{n}{n-1}\binom{k+n-1}{n}-d^n
 \right)<0.
\end{equation}
Thus the sequence is initially strictly log-convex.
If $n=2$, then $\Delta_k>0$ for every $k\geq d$.
If $n\geq3$, then
\begin{equation}\label{eq:late-lc-range}
 \Delta_k>0\quad\text{for every }k\geq K=(n-1)(d-1).
\end{equation}
\end{proposition}

\begin{proof}
If $d=1$, \eqref{eq:tail-formula} gives
\[
 \mu_k=\binom{n+k}{n}.
\]
The ratio $\mu_{k+1}/\mu_k=(n+k+1)/(k+1)$ is strictly decreasing.
This proves strict log-concavity.

Assume $d\geq2$.
For $0\leq k\leq d$, only the term $j=1$ occurs in
\eqref{eq:binomial-formula}. Hence
\begin{equation}\label{eq:early-mu-n}
 \mu_k=d^n+n\binom{k+n-1}{n}.
\end{equation}
A direct calculation using two consecutive binomial identities gives
\eqref{eq:early-defect-n}.
We prove that the factor in parentheses is negative.
It increases with $k$, so it is enough to use $k=d-1$.
Now
\begin{align*}
 \binom{d+n-2}{n}
 &=\frac{d-1}{n}\binom{d+n-2}{n-1}\\
 &\leq\frac{d-1}{n}d^{n-1}
 <\frac{d^n}{n}.
\end{align*}
The middle inequality follows from
\[
 \binom{d+n-2}{n-1}
 =\prod_{i=1}^{n-1}\frac{d-1+i}{i}\leq d^{n-1}.
\]
This proves \eqref{eq:early-defect-n}.

Let
\[
 B_k=\binom{n+d+k-1}{n}.
\]
By \Cref{cor:sharp-stabilization}, $\mu_k=B_k$ for $k\geq K$.
The ratios $B_{k+1}/B_k$ are strictly decreasing.
Thus $\Delta_k>0$ for every $k\geq K+1$.

We still need the boundary case $k=K$.
By \eqref{eq:tail-formula} and the one-dimensional socle,
\[
 \mu_{K-1}=B_{K-1}+1,\qquad
 \mu_K=B_K,\qquad
 \mu_{K+1}=B_{K+1}.
\]
One finds that $\Delta_K>0$ is equivalent to
\begin{equation}\label{eq:boundary-binomial-ineq}
 \binom{nd}{n}>d(nd+1).
\end{equation}
For $n=2$, this inequality is false.
In this case $K=d-1$, and the strict binomial argument starts at $k=d$.

Suppose $n\geq3$.
Since
\[
 \binom{nd}{n}=d\binom{nd-1}{n-1},
\]
it is enough to prove $\binom{nd-1}{n-1}>nd+1$.
Here $2\leq n-1\leq(nd-1)/2$.
Therefore
\[
 \binom{nd-1}{n-1}\geq\binom{nd-1}{2}>nd+1.
\]
This proves \eqref{eq:boundary-binomial-ineq} and the proposition.
\end{proof}

The theorem leaves a finite middle window
\[
 d\leq k<(n-1)(d-1)
\]
when $n\geq3$.
The defects in this window need not be monotone.

\begin{example}\label{ex:quartic-point-defects}
For $n=3$ and $r=4$, we have $d=3$ and $K=4$.
The universal sequence begins
\[
 27,30,39,57,84,120,165,\ldots.
\]
The first four defects are
\[
 -153,\ -189,\ -27,\ 216.
\]
This example shows why monotonicity of the defects is not the right proof method.
It also displays the initial log-convex phase and the log-concave tail.
\end{example}

We now drop the ordinary semi-homogeneous assumption.

\begin{proposition}\label{prop:eventual-lc}
Let \(n\geq2\), and let \(f\in\mathcal O_n\) have an isolated critical point.
Then
\begin{equation}\label{eq:mu-asymptotic}
 \mu_k(f)=\frac{k^n}{n!}+O(k^{n-1}).
\end{equation}
If \(r=\ord(f)\), then
\begin{equation}\label{eq:tau-asymptotic}
 \tau_k(f)=\frac{r k^{n-1}}{(n-1)!}+O(k^{n-2}).
\end{equation}
Both sequences agree with polynomials for all large $k$.
They are eventually strictly log-concave.
\end{proposition}

\begin{proof}
By \eqref{eq:mu-dictionary}, $\mu_k-\mu_0$ is the Hilbert--Samuel function
$\length_R(J/\m^kJ)$.
The ideal \(J\) has rank one over the regular local domain \(R\).
The associativity formula for multiplicities therefore gives
\[
 e_{\m}(J)=e_{\m}(R)=1.
\]
This proves \eqref{eq:mu-asymptotic}.

For the Tjurina sequence, use the notation of \Cref{prop:layer-dictionary}.
The ring $A=R/(f)$ has dimension $n-1$ and multiplicity $r$.
The quotient \(A/\bar J\) has finite length.  Since \(\dim A=n-1>0\),
additivity of multiplicity in
\[
 0\longrightarrow\bar J\longrightarrow A
 \longrightarrow A/\bar J\longrightarrow0
\]
gives \(e_{\bar\m}(\bar J)=e_{\bar\m}(A)=r\).
Equation \eqref{eq:tau-dictionary} gives \eqref{eq:tau-asymptotic}.
The polynomial statements are the Hilbert--Samuel theorem; see, for example,
\cite{Matsumura1986,BrunsHerzog1993}.

Finally, let $P(k)=ak^e+O(k^{e-1})$ be a polynomial with $a>0$ and $e\geq1$.
A direct expansion gives
\[
 P(k)^2-P(k-1)P(k+1)
 =e a^2k^{2e-2}+O(k^{2e-3}).
\]
This is positive for all large $k$.
Apply this to the two Hilbert--Samuel polynomials.  Together with
\Cref{prop:two-log-phases}, this proves \Cref{thm:main-c}.
\end{proof}

\section{The complete plane transition}\label{sec:plane}

We now prove \Cref{thm:main-d}.

Let $f\in\C\{x,y\}$ define an isolated plane curve singularity.
Write
\[
 r=\ord(f),\qquad \mu=\mu_0(f).
\]
We use the following formula of Fern\'andez-P\'erez and Ribeiro as an
external input \cite[Corollary~4.1]{FernandezPerezRibeiro2025}:
\begin{equation}\label{eq:plane-mu-formula}
 \mu_k=
 \begin{cases}
  \mu+k(k+1),&0\leq k<r,\\[3pt]
  \displaystyle \mu+k(k+1)-\binom{k-r+2}{2},&k\geq r.
 \end{cases}
\end{equation}
With the usual zero convention, the two expressions also agree at
\(k=r-1\).  The formula already shows that the full sequence is
determined by \((\mu,r)\).  We now use it to compute every
log-concavity defect and determine its sign.

For completeness, we give a direct proof of
\eqref{eq:plane-mu-formula}.  Put \(I=J(f)=(a,b)\).  After interchanging
\(x\) and \(y\), we may assume that
\[
 \ord(a)=r-1\leq\ord(b).
\]
Since \(f\) has an isolated critical point, \(a,b\) form a regular
sequence.  Hence
\[
 0\longrightarrow R
 \xrightarrow{\ c\mapsto c(-b,a)\ }
 R^2\longrightarrow I\longrightarrow0
\]
is exact.  The image of \(\m^kR^2\) is \(\m^kI\).  The first map becomes
zero in \((R/\m^k)^2\) when \(k<r\).  When \(k\geq r\), its image is
isomorphic to \(R/\m^{k-r+1}\).  Therefore
\[
 \length_R\frac{I}{\m^kI}
 =2\binom{k+1}{2}-\binom{k-r+2}{2}
 =k(k+1)-\binom{k-r+2}{2},
\]
where the last binomial coefficient is zero for \(k<r\).  Finally, the
exact sequence
\[
 0\longrightarrow I/\m^kI
 \longrightarrow R/\m^kI
 \longrightarrow R/I
 \longrightarrow0
\]
gives \eqref{eq:plane-mu-formula}.

\begin{proposition}\label{prop:plane-defect}
For $k\geq1$, let $\Delta_k=\mu_k^2-\mu_{k-1}\mu_{k+1}$.
Then
\begin{equation}\label{eq:plane-defect-formula}
 \Delta_k=
 \begin{cases}
  2\bigl(k(k+1)-\mu\bigr),&1\leq k\leq r-2,\\[3pt]
  \displaystyle\frac12\left[
  k^2+(2r-1)k+(3r-2)(r-1)-2\mu
  \right],&k\geq r-1.
 \end{cases}
\end{equation}
The first range is strictly negative.
The second expression is strictly increasing in $k$.
Consequently, the sequence has at most one change from strict log-convexity to log-concavity.
More precisely, the first index with $\Delta_k\geq0$ is
\begin{equation}\label{eq:plane-transition-index}
 \kappa(f)=r-1+\min\left\{s\geq0:
 (r-1)(3r-2)+\frac{s(s+4r-3)}2\geq\mu
 \right\}.
\end{equation}
There is equality $\Delta_{\kappa(f)}=0$ exactly when equality holds inside the braces.
All later defects are positive.
\end{proposition}

\begin{proof}
Substitute \eqref{eq:plane-mu-formula} into the definition of $\Delta_k$.
This gives the two expressions in \eqref{eq:plane-defect-formula}.
The classical multiplicity bound
\begin{equation}\label{eq:plane-milnor-lower}
 \mu\geq(r-1)^2
\end{equation}
implies that the first expression is negative.
Indeed,
\[
 \mu=i_0(f_x,f_y)
 \geq\ord(f_x)\ord(f_y)
 \geq(r-1)^2,
\]
where the first inequality is the local B\'ezout inequality.
This is also a special case of Teissier's mixed-multiplicity inequalities
\cite{Teissier1977Minkowski}.
Indeed, for $k\leq r-2$,
\[
 k(k+1)\leq(r-2)(r-1)< (r-1)^2.
\]

The second expression has positive first difference for $k\geq r-1$.
Write $k=r-1+s$.
Then the condition $\Delta_k\geq0$ becomes
\[
 (r-1)(3r-2)+\frac{s(s+4r-3)}2\geq\mu.
\]
This proves the formula for $\kappa(f)$ and all sign statements.
\end{proof}

\begin{corollary}\label{cor:plane-global-lc}
The full sequence $(\mu_k)_{k\geq0}$ is log-concave at every positive index if and only if
\[
 r=2\quad\text{and}\quad \mu\leq4.
\]
It is strictly log-concave at every positive index if and only if
\[
 r=2\quad\text{and}\quad \mu\leq3.
\]
\end{corollary}

\begin{proof}
If $r\geq3$, the first range in \eqref{eq:plane-defect-formula} is nonempty and negative.
If $r=2$, the second formula gives $\Delta_1=4-\mu$ and then increases strictly.
\end{proof}

\begin{example}\label{ex:A8}
For $f=x^2+y^9$, one has $(r,\mu)=(2,8)$.
Formula \eqref{eq:plane-mu-formula} gives
\[
 (\mu_0,\mu_1,\mu_2,\mu_3,\mu_4)=(8,10,13,17,22).
\]
For a direct certificate, \(J=(x,y^8)\); bases of the first three
quotients are
\[
\begin{aligned}
 R/J &: 1,y,\ldots,y^7,\\
 R/\m J &: 1,y,\ldots,y^8,x,\\
 R/\m^2J &: 1,y,\ldots,y^9,x,xy,x^2.
\end{aligned}
\]
Thus
\[
 \Delta_1=10^2-8\cdot13=-4,
 \qquad
 \Delta_2=13^2-10\cdot17=-1.
\]
The values \(\mu_1,\mu_2,\mu_3\) also agree with
\cite[Table~2]{HussainZuo2026}.  This shows that the
higher-Milnor part of Conjecture~1.10 does not hold. %
\end{example}

\begin{example}\label{ex:E6-tau}
Let $f=x^3+y^4$, of type $E_6$.
In \(A=\C\{x,y\}/(x^3+y^4)\), which is free over \(\C\{y\}\) with
basis \(1,x,x^2\), the first three relevant ideals are
\[
\begin{aligned}
 \bar\m J&=(y^4,x^2y,xy^3),\\
 \bar\m^2J&=(y^5,xy^4,x^2y^2),\\
 \bar\m^3J&=(y^6,xy^5,x^2y^3).
\end{aligned}
\]
The corresponding quotient bases have respectively \(8,11,14\) elements.
Together with the ordinary Tjurina basis of size \(6\), this gives
\[
 (\tau_0,\tau_1,\tau_2,\tau_3)=(6,8,11,14).
\]
Hence
\[
 \tau_1^2-\tau_0\tau_2=64-66=-2.
\]
Thus the higher-Tjurina parts of Conjecture~1.10 in
\cite{HussainZuo2026}  do not hold. %
 Together, the two examples show that eventual log-concavity does not control the first layers.
\end{example}

The first statement below is already contained in the plane formula.
The converse statement and the failure of contact classification are
further consequences.

\begin{corollary}\label{cor:plane-information}
For an isolated plane germ, the full sequence $(\mu_k)_{k\geq0}$ is determined by $(\mu,r)$.
Conversely, the sequence determines $(\mu,r)$.
It does not determine the contact-equivalence class.
\end{corollary}

\begin{proof}
The first statement is \eqref{eq:plane-mu-formula}.
The value $\mu$ is $\mu_0$.
For \(k\geq1\), set
\[
 c_k=\mu_{k+1}-2\mu_k+\mu_{k-1}.
\]
Formula \eqref{eq:plane-mu-formula} gives
\[
 c_k=2\quad(1\leq k\leq r-2),
 \qquad
 c_k=1\quad(k\geq r-1).
\]
Hence
\[
 r=1+\min\{k\geq1:c_k=1\}.
\]

For the last statement, compare
\[
 f_D=x^2y+y^5\quad\text{and}\quad f_E=x^3+y^4.
\]
Both have $(\mu,r)=(6,3)$, so all their higher Milnor numbers agree.
Their degree-three tangent cones are $x^2y$ and $x^3$.
These binary cubics are not linearly equivalent.
Since contact equivalence preserves the tangent cone up to a nonzero scalar and a linear change of variables, the two germs are not contact equivalent.
Together with \Cref{prop:plane-defect,cor:plane-global-lc}, this proves
\Cref{thm:main-d}.
\end{proof}

\section{Projective moduli and information loss}\label{sec:geometry-loss}

The results above show two forms of rigidity.
In the plane, the sequence depends only on $(\mu,r)$.
For an ordinary multiple point in any dimension, it depends only on $(n,r)$.
We now explain the geometry that these numbers miss.

Suppose two germs $f$ and $g$ of multiplicity $r$ are contact equivalent.
Then
\[
 f=u\cdot(g\circ\phi)
\]
for a unit $u$ and a local coordinate change $\phi$.
If $L$ is the linear part of $\phi$, their lowest homogeneous terms satisfy
\[
 f_r=u(0)\,g_r\circ L.
\]
Thus contact equivalence of the germs forces projective equivalence of their tangent-cone hypersurfaces.

\begin{proposition}
\label{prop:projective-moduli-invisible}
Fix $n$ and $r$.
Let $X$ and $X'$ be smooth degree-$r$ hypersurfaces in $\mathbb P^{n-1}$ that are not projectively equivalent.
Let $f_r$ and $f'_r$ be their homogeneous equations, viewed as germs at the vertices of the affine cones.
Then the germs are not contact equivalent, but
\[
 \mu_k(f_r)=\mu_k(f'_r)\quad\text{for every }k\geq0.
\]
The same statements remain true after arbitrary terms of degree greater
than \(r\) are added to either equation.
\end{proposition}

\begin{proof}
The tangent-cone argument above gives non-equivalence.
The equality follows from \Cref{thm:main-a}.
\end{proof}

\begin{example}\label{ex:plane-quartic-moduli}
The projective equivalence classes of smooth plane quartics form a
six-dimensional moduli space.
Their affine cones are ordinary quartic points in \(\C^3\).  If two such
quartics are not projectively equivalent, their affine-cone germs are
not contact equivalent.  Nevertheless, all these germs have
\[
 (\mu_k)_{k\geq0}=(27,30,39,57,84,120,165,\ldots).
\]
The sequence is therefore constant on this six-dimensional moduli space,
so it does not detect the moduli parameter.
\end{example}

\begin{example}\label{ex:Hesse-pencil}
Consider
\[
 f_\lambda=x^3+y^3+z^3-3\lambda xyz,
 \qquad \lambda^3\neq1.
\]
The projective curves $V_+(f_\lambda)$ are smooth members of the Hesse pencil.
Their $j$-invariant varies with $\lambda$; see \cite{ArtebaniDolgachev2009}.
Choose two parameters with different $j$-invariants.
The two cone germs are not contact equivalent, yet both have
\[
 (\mu_k)_{k\geq0}=(8,11,20,35,56,\ldots).
\]
This is a one-parameter version of \Cref{prop:projective-moduli-invisible}.
\end{example}

The higher terms of a germ are also invisible in this class.

\begin{example}\label{ex:higher-term-direction}
Let
\[
 f_t=x^4+y^4+z^4+t x^2y^2z.
\]
The last monomial has degree five, so every $f_t$ has the same smooth quartic tangent cone.
By \Cref{thm:main-a}, every $\mu_k(f_t)$ is independent of $t$.
At $t=0$, the class of $x^2y^2z$ is nonzero in
\[
 \frac{\C\{x,y,z\}}{J(f_0)}
 =\frac{\C\{x,y,z\}}{(x^3,y^3,z^3)}.
\]
Hence this family gives a nonzero first-order right-deformation direction.
More precisely, this is the extended-right convention; for coordinate
changes fixing the origin, the tangent ideal is $\m J(f_0)$, and the class
remains nonzero there as well.  Thus the displayed direction is not
infinitesimally trivial under either convention, although all numerical
higher Milnor invariants stay fixed. %
For deformation-theoretic background, see \cite{GreuelLossenShustin2007}.
\end{example}

For fixed \((n,r)\), the dimensions of the graded pieces of
\[
 A_X=S/J(f_r)
\]
are the same for every smooth tangent cone \(X\).  The ring also has
multiplication maps
\[
 (A_X)_a\otimes_{\C}(A_X)_b\longrightarrow(A_X)_{a+b}.
\]
Formula \eqref{eq:tail-formula} uses only the dimensions
\(h_a=\dim_{\C}(A_X)_a\).  It does not retain these multiplication maps.

\section*{AI disclosure}

ChatGPT was used to assist with checking mathematical arguments and improving the exposition and language. The authors take full responsibility for the mathematical content and final presentation of the paper.

\bibliographystyle{alpha}
\bibliography{references}

\end{document}